\documentclass[12pt]{article}
\usepackage{amsmath,amsfonts,amsthm,amssymb, mathtools}
\usepackage{mathrsfs, graphicx,color,latexsym, tikz, calc,cite,enumerate,indentfirst}
\usetikzlibrary{shadows}
\usetikzlibrary{patterns,arrows,decorations.pathreplacing}
\usepackage[colorlinks=true,
linkcolor=blue,citecolor=blue,
urlcolor=blue]{hyperref}

\newtheorem{theorem}{Theorem}[section]
\newtheorem{lemma}{Lemma}[section]
\newtheorem{problem}{Problem}[section]

\newtheorem{corollary}{Corollary}[section]

\newtheorem{example}{Example}
\newtheorem{claim}{Claim}

\renewcommand\proofname{\it{Proof}}

\allowdisplaybreaks

\title{\bf Distance-regular Cayley graphs over $\mathbb{Z}_{p^s}\oplus\mathbb{Z}_{p}$}

\author{Xiongfeng Zhan$^a$, Lu Lu$^b$, Xueyi Huang$^{a,}$\thanks{Corresponding author.}\setcounter{footnote}{-1}\footnote{\emph{Email address:} zhanxfmath@163.com (X. Zhan), lulumath@csu.edu.cn (L. Lu), huangxymath@163.com (X. Huang).}\\[2mm]
	\small $^a$School of Mathematics, East China University of Science and Technology, \\
	\small  Shanghai 200237, P. R. China\\
	\small $^b$School of Mathematics and Statistics, Central South University,\\
	\small Changsha, Hunan, 410083, P. R. China
}

\date{}

\begin{document}
	\maketitle
	
	\begin{abstract}
		In [Distrance-regular Cayley graphs on dihedral groups, J. Combin. Theory Ser B 97 (2007) 14--33], Miklavi\v{c} and Poto\v{c}nik proposed the  problem of characterizing distance-regular Cayley graphs, which can be viewed as an extension of the problem of identifying strongly regular Cayley graphs, or equivalently, regular partial difference sets. In this paper, all  distance-regular Cayley graphs over $\mathbb{Z}_{p^s}\oplus\mathbb{Z}_{p}$ with $p$ being an odd prime are determined. It is shown that every such graph   is isomorphic to a complete graph, a complete multipartite graph, or the line graph of a transversal design $TD(r,p)$ with $2\leq r\leq p-1$.
		
		\par\vspace{2mm}
		
		\par\vspace{1mm}
		
	\end{abstract}

	\section{Introduction}\label{section::1}
	
	Let $G$ be a finite group with identity $1$, and let $S$ be an inverse closed subset of $G\setminus\{1\}$. The \textit{Cayley graph} $\mathrm{Cay}(G,S)$ is the graph with vertex set $G$, and with an edge joining two vertices $g,h\in G$ if and only if $g^{-1}h\in S$. Here $S$ is called the \textit{connection set} of $\mathrm{Cay}(G,S)$. It is known that $\mathrm{Cay}(G,S)$ is connected if and only if  $\langle S \rangle=G$, and that $G$ acts regularly on the vertex set of $\mathrm{Cay}(G,S)$ by left multiplicity. 
	
	Let $\Gamma$ be a connected graph with vertex set $V(\Gamma)$. The \textit{distance} $\partial_\Gamma(u,v)$ between two vertices $u,v$ of $\Gamma$ is the length of a shortest path connecting them in $\Gamma$, and the \textit{diameter} $d_\Gamma$ of $\Gamma$ is the maximum distance  in $\Gamma$. For $v\in V(\Gamma)$, let $N_i^\Gamma(v)$ denote the set of vertices at distance $i$ from $v$ in $\Gamma$. In particular, we denote $N^\Gamma(v)=N_1^\Gamma(v)$. When $\Gamma$ is clear from the context, we use  $\partial(u,v)$, $d$, $N_i(v)$ and $N(v)$ instead of  $\partial_\Gamma(u,v)$, $d_\Gamma$, $N_i^\Gamma(v)$ and $N^\Gamma(v)$, respectively. For  $u,v\in V(\Gamma)$ with $\partial(u,v)=i$ ($0\leq i\leq d$), let
	$$
	c_i(u,v)=|N_{i-1}(u)\cap N(v)|,~~ a_i(u,v)=|N_{i}(u)\cap N(v)|, ~~ b_i(u,v)=|N_{i+1}(u)\cap N(v)|.
	$$
	Here  $c_0(u,v)=b_d(u,v)=0$. The graph $\Gamma$ is called \textit{distance-regular} if  $c_i(u,v)$, $b_i(u,v)$ and $a_i(u,v)$ only depend on the distance $i$ between $u$ and $v$ but not the choice of $u,v$.
	
	For a distance-regular graph $\Gamma$ with diameter $d$, we denote  $c_i=c_i(u,v)$, $a_i=a_i(u,v)$ and $b_i=b_i(u,v)$, where $u,v\in V(\Gamma)$ with $\partial(u,v)=i$. We also denote $k_i=|N_i(v)|$, where $v\in V(\Gamma)$. Clearly, $k_i$ is independent of the choice of $v$. By definition, $\Gamma$ is a regular graph with valency $k=b_0$, and
	$a_i+b_i+c_i=k$ for $0\leq i\leq d$.  The array $\{b_0,b_1,\ldots,b_{d-1};c_1,c_2,\ldots,c_d\}$ is called the \textit{intersection array} of $\Gamma$. In particular, $\lambda=a_1$ is the number of common neighbors between two adjacent vertices in $\Gamma$, and $\mu=c_2$ is the number of common neighbors between two vertices  at distance $2$ in $\Gamma$.  A distance-regular graph on $n$ vertices with valency $k$ and diameter $2$ is also called a  \textit{strongly regular graph} with parameters $(n,k,\lambda=a_1,\mu=c_2)$. 
	
	The concept of  distance-regular graph was introduced by Biggs  \cite{B74}. In the past several decades, distance-regular graphs played an important role in the study of design theory and coding theory, and were closely linked to some other subjects such as finite group theory, finite
	geometry, representation theory, and association schemes. Moreover, distance-regular graphs have applications in several fields  such as (quantum) information
	theory, diffusion models, (parallel) networks, and even finance. For systematic studies on  distance-regular graphs, we refer the reader to the famous monograph  \cite{BCN89} and the nice survey paper \cite{DKT16}.
	
	As a subclass of distance-regular graphs, strongly regular graphs have aroused a lot of interest, and the
	subject concerns beautifully regular structures, studied mostly using spectral
	methods, group theory, geometry and sometimes lattice theory \cite{BV22}.  It is known that  strongly regular Cayley graphs  are equivalent to regular partial difference sets, while partial difference sets are closely related to association schemes of partially
	balanced incomplete block designs with two associate classes  \cite{M84,M94}.  As we know,  strongly regular Cayley graphs over cyclic groups  were determined by Bridges and Mena \cite{BM79}, Ma \cite{M84}, and partially by Maru\v{s}i\v{c} \cite{M89}, and strongly regular Cayley graphs over $\mathbb{Z}_{p^n}\oplus \mathbb{Z}_{p^n}$ ($p$ is an odd prime) were classified by Leifman and Muzychuck \cite{LM05}. However, the strongly regular Cayley graphs over general groups, even for abelian groups, are far from being completely characterized.
	
	As an extension of the problem of characterizing strongly regular Cayley graphs, Miklavi\v{c} and Poto\v{c}nik  \cite{MP07} (see also \cite[Problem 71]{DKT16}) proposed the following general problem.
	\begin{problem}\label{prob::main}
		For a class of groups $\mathcal{G}$, determine all distance-regular graphs, which are Cayley graphs on a group in $\mathcal{G}$.
	\end{problem}
	
	Up to now, there have been some  advances on Problem \ref{prob::main}. Miklavi\v{c} and Poto\v{c}nik \cite{MP03,MP07} (almost) classified the distance-regular Cayley graphs over cyclic groups and dihedral groups. Miklavi\v{c} and \v{S}parl \cite{MS14,MS20} characterized the distance-regular Cayley graphs  on abelian groups or generalized dihedral groups under the condition that the connection set is minimal with respect to some element. Abdollahi, van Dam and Jazaeri \cite{ADJ17} determined the distance-regular Cayley graphs of diameter at most three with least eigenvalue $-2$.  van Dam and Jazaeri \cite{DJ19,DJ21} determined some distance-regular Cayley graphs with small valency, and provided some characterizations for  bipartite distance-regular Cayley graphs with diameter  $3$ or $4$. Very recently, Huang, Das and Lu \cite{HD22,HDL23} gave a partial characterization of distance-regular Cayley graphs over dicyclic groups, and determined all distance-regular Cayley graphs over generalized dicyclic groups under the condition that the connection set is minimal with respect to some element.
	
	In this paper, we mainly consider Problem \ref{prob::main} for   $p$-groups with a cyclic subgroup of index $p$. More specifically, we  determine all distance-regular Calyley graphs over $\mathbb{Z}_{p^s}\oplus\mathbb{Z}_{p}$, where $p$ is an odd prime and $s\geq 1$. The main result is as follows.

	\begin{theorem}\label{thm::main}
		Let $p$ be an odd prime, and let  $\Gamma$ be a Cayley graph over $\mathbb{Z}_{p^s}\oplus\mathbb{Z}_{p}$ with $s\geq 1$. Then $\Gamma$ is distance-regular if and only if it is isomorphic to one of the following graphs:
		\begin{enumerate}[$(i)$]\setlength{\itemsep}{0pt}
			\item the complete graph $K_{p^{s+1}}$;
			\item the complete multipartite graph $K_{t\times m}$ with  $tm=p^{s+1}$, which is the complement of the union of $t$ copies of $K_m$;
			\item the graph $\mathrm{Cay}(\mathbb{Z}_{p}\oplus\mathbb{Z}_{p},S)$ with $S=\cup_{i=1}^rH_i\setminus\{(0,0)\}$, where $2\leq r\leq p-1$, and $H_i$ ($i=1,\ldots,r$) are subgroups of order $p$ in $\mathbb{Z}_{p}\oplus\mathbb{Z}_{p}$.
		\end{enumerate}
		In particular, the graph in (iii) is the line graph of a transversal design $TD(r,p)$, which is a strongly regular graph with parameters $(p^2,r(p-1),p+r^2-3r,r^2-r)$.
	\end{theorem}

	\section{Preliminaries}\label{section::2}
	
	In this section, we review some notations and results about distance-regular graphs, which are powerful in the proof of Theorem \ref{thm::main}.

	Let $\Gamma$ be a graph, and let $\mathcal{B}=\{B_1,\ldots,B_\ell\}$ be a partition of $V(\Gamma)$ (here $B_i$ are called \textit{blocks}). The \textit{quotient graph} of $\Gamma$ with respect to $\mathcal{B}$, denoted by $\Gamma_\mathcal{B}$, is the graph with vertex set $\mathcal{B}$, and with  $B_i,B_j$ ($i\neq j$) adjacent if and only if there exists at least one edge between  $B_i$ and $B_j$ in $\Gamma$. Moreover, we say that $\mathcal{B}$ is an \textit{equitable partition} of $\Gamma$ if there are integers $b_{ij}$ ($1\leq i,j\leq \ell$) such that every vertex in $B_i$ has exactly $b_{ij}$ neighbors in $B_j$. In particular, if every block of $\mathcal{B}$ is an independent set, and between any two blocks there are either no edges or there is a perfect matching, then $\mathcal{B}$ is an equitable partition of  $\Gamma$. In this situation,  $\Gamma$ is called a \textit{cover} of its quotient graph $\Gamma_\mathcal{B}$, and the blocks are called \textit{fibres}. If $\Gamma_\mathcal{B}$ is connected, then all fibres have the same size, say $r$, called \textit{covering index}.

	A graph $\Gamma$ of diameter $d$ is \textit{antipodal} if the relation $\mathcal{R}$ on $V(\Gamma)$ defined by $u\mathcal{R}v\Leftrightarrow\partial(u,v)\in\{0,d\}$ is an equivalence relation, and the corresponding equivalence classes are called \textit{antipodal classes}. A cover of index $r$, in which the fibres are antipodal classes, is called \textit{an $r$-fold  antipodal cover} of its quotient.

	Suppose that $\Gamma$ is a distance-regular graph with diameter $d$. For  $i\in\{1,\ldots,d\}$, the \textit{$i$-th distance graph} $\Gamma_i$  is  the graph with vertex set  $V(\Gamma)$ in which two distinct vertices are adjacent if and only if they are at distance $i$ in $\Gamma$. It is  known that an imprimitive distance-regular graph with valency at least $3$ is either bipartite, antipodal, or both \cite[Theorem 4.2.1]{BCN89}. If $\Gamma$ is a bipartite distance-regular graph, then $\Gamma_2$ has two connected components, which are called the \textit{halved graphs} of $\Gamma$ and denoted by $\Gamma^+$ and $\Gamma^-$. For convenience, we use  $\frac{1}{2}\Gamma$ to represent any one of these two graphs. If $\Gamma$ is an antipodal distance-regular graph, then all antipodal classes have the same size, say $r$, and form an equitable partition $\mathcal{B}^\ast$ of $\Gamma$. The quotient graph $\overline{\Gamma}:=\Gamma_{\mathcal{B}^\ast}$ is called the \textit{antipodal quotient} of $\Gamma$. If $d=2$, then $\Gamma$ is a complete multipartite graph. If $d\geq 3$, then the edges between  two distinct antipodal classes of $\Gamma$ form an empty set or a perfect matching. Thus $\Gamma$ is an $r$-fold antipodal  cover of $\overline{\Gamma}$ with the antipodal classes as its fibres. 
	
	%the  relation $\mathcal{R}$ defined  above leads to a partition $\mathcal{B}^\ast$ of $V(\Gamma)$ into equivalence classes, called \textit{fibres}.  It is known that  $\mathcal{B}^\ast$ is actually  an equitable partition of $\Gamma$, and all fibres of $\Gamma$ share the same size.  The  \textit{antipodal quotient} of $\Gamma$, denoted by $\overline{\Gamma}$, is defined as the quotient graph  $\Gamma_{\mathcal{B}^\ast}$. Let $r$ be the common size of fibres of $\Gamma$. Then  $\Gamma$ is said to be an $r$-\textit{fold antipodal cover} of $\overline{\Gamma}$. Note that if $d=2$ then $\Gamma$ is a complete multipartite graph, and that if $d\geq 3$ then the edges between  two distinct fibres of $\Gamma$ form an empty set or a perfect matching.
	
	\begin{lemma}(\cite[Proposition 4.2.2]{BCN89}) \label{lem::imprimitive}
		Let $\Gamma$ denote an imprimitive distance-regular graph with diameter $d$ and valency $k \geq 3$. Then the following hold.
		\begin{enumerate}[$(i)$]\setlength{\itemsep}{0pt}
			\item If $\Gamma$ is bipartite, then the halved graphs of $\Gamma$ are non-bipartite distance-regular graphs with diameter $\lfloor\frac{d}{2}\rfloor$.
			\item If $\Gamma$  is antipodal, then $\overline{\Gamma}$ is a distance-regular graph with diameter $\lfloor\frac{d}{2}\rfloor$.
			\item  If $\Gamma$ is antipodal, then $\overline{\Gamma}$ is not antipodal, except when $d\leq 3$  (in that case $\overline{\Gamma}$  is a complete graph), or when $\Gamma$ is bipartite with $d = 4$ (in that case $\overline{\Gamma}$ is a complete bipartite graph).
			
			\item If $\Gamma$ is antipodal and has odd diameter or is not bipartite, then $\overline{\Gamma}$ is primitive.
			\item If $\Gamma$ is bipartite and has odd diameter or is not antipodal, then the halved graphs of $\Gamma$ are primitive.
			\item If $\Gamma$ has even diameter and is both bipartite and antipodal, then $\overline{\Gamma}$ is bipartite. Moreover, if $\frac{1}{2}\Gamma$ is a halved graph of $\Gamma$, then it is antipodal, and $\overline{\frac{1}{2}\Gamma}$ is primitive and isomorphic to $\frac{1}{2}\overline{\Gamma}$.
		\end{enumerate}
	\end{lemma}
	The following lemma is an immediate  consequence of \cite[Proposition 5.1.1]{BCN89}.
	\begin{lemma}\label{lem::k1=k2}
		Let $\Gamma$ denote a distance-regular graph with diameter $d$ and valency $k>2$. If $k_2<k_1$ then $d\leq2$. Moreover, if $k_2=k_1$, then either $d=2$, or $d=3$ and $X$ is a $2$-fold antipodal  cover of the complete graph.  
	\end{lemma}
	
	\begin{lemma}(\cite[Theorem 6.2]{GH92})\label{lem::quotient_graph}
		Let $\Gamma$ be an antipodal distance-regular graph with diameter $d\geq 3$, and let $\mathcal{B}$ be an equitable partition of $\Gamma$ with each block contained in a
		fibre of $\Gamma$. Assume that no block of $\mathcal{B}$ is a single vertex, or a fibre. Then all
		blocks  of  $\mathcal{B}$ have the same size, and the  quotient graph $\Gamma_{\mathcal{B}}$ is an
		antipodal distance-regular graph with diameter $d$. Moreover, $\Gamma$ and $\Gamma_{\mathcal{B}}$ have isomorphic antipodal quotients.
	\end{lemma}
	
	\begin{lemma}(\cite[p. 431]{BCN89})\label{lem::antipodal_DRG}
		Let $\Gamma$ be a non-bipartite $r$-fold antipodal distance-regular graph on $v$ vertices with  diameter $3$ and valency $k$. Then $v=r(k+1)$, $k=\mu(r-1)+\lambda+1$, and $\Gamma$ has  the intersection array $\{k,\mu(r-1),1;1,\mu,k\}$ and the spectrum $\{k^1,\theta_1^{m_1},\theta_2^k,\theta_3^{m_3}\}$, where
		\begin{equation*}
			\theta_1=\frac{\lambda-\mu}{2}+\delta,~~\theta_2=-1,~~\theta_3=\frac{\lambda-\mu}{2}-\delta,~~\delta=\sqrt{k+\left(\frac{\lambda-\mu}{2}\right)^2},
		\end{equation*}
		and
		\begin{equation*}
			m_1=-\frac{\theta_3}{\theta_1-\theta_3}(r-1)(k+1),~~m_3=\frac{\theta_1}{\theta_1-\theta_3}(r-1)(k+1).
		\end{equation*}
		Moreover, if $\lambda\neq\mu$, then all eigenvalues  of $\Gamma$ are integers.
	\end{lemma}

	A strongly regular graph with parameters $(n,k=\frac{n-1}{2},\lambda=\frac{n-5}{4},\mu=\frac{n-1}{4})$  ($n\equiv 1\pmod 4$) is called a \textit{conference graph}.  Let $\mathbb{F}_q$ denote the finite field of order $q$  where $q\equiv 1\pmod 4$ is a prime power. The \textit{Paley graph} $P(q)$ is defined as the graph with vertex set $\mathbb{F}_q$ in which two distinct vertices $u,v$ are adjacent if and only if $u-v$ is a square in the multiplicative group of $\mathbb{F}_q$. It is known that Paley graphs are comference graphs \cite{S62,ER63}.

	\begin{lemma}(\cite[p. 180]{BCN89})\label{lem::confer}
		Let $\Gamma$ be a conference graph (or particularly, Paley graph). Then $\Gamma$ has no distance-regular $r$-fold antipodal covers for $r>1$, except for the pentagon $C_5\cong P(5)$, which is covered by the decagon $C_{10}$.
	\end{lemma}

	Recall that circulants are Cayley graphs on cyclic groups. In \cite{MP03}, Miklavi\v{c} and Poto\v{c}nik determined all distance-regular circulants.
	
	\begin{lemma}(\cite[Theorem 1.2, Corollary 3.7]{MP03}) \label{lem::cir_DRG}
		Let $\Gamma$ be a circulant on $n$ vertices. Then $\Gamma$ is distance-regular if and only if it is isomorphic to one of the following graphs:
		\begin{enumerate}[$(i)$]\setlength{\itemsep}{0pt}
			\item the cycle $C_n$;
			\item the complete graph $K_n$;
			\item the complete multipartite graph $K_{t\times m}$, where $tm=n$;
			\item the complete bipartite graph without a perfect matching $K_{m,m}-mK_2$, where $2m = n$, $m$ odd;
			\item the Paley graph $P(n)$, where $n\equiv 1\pmod 4$  is prime.
		\end{enumerate}
		In particular, $\Gamma$ is a primitive distance-regular graph if and only if $\Gamma\cong K_n$, or $n$ is prime, and $\Gamma\cong C_n$ or $P(n)$.
	\end{lemma}
	
	Let $G$ be a group, and let $\mathbb{Z}G$ be the group algebra of $G$ over the ring of  integers $\mathbb{Z}$. For an integer $m$ and an element $a=\sum_{g\in G} a_g g\in \mathbb{Z}G$, we define
	$$a^{(m)}=\sum_{g\in G} a_g g^m\in \mathbb{Z}G.$$
	Also, for a subset $T\subseteq G$, we denote 
	$$
	\underline{T}=\sum_{g\in T}g\in \mathbb{Z}G.
	$$
	If there is a  partition $\{T_0,T_1,\ldots,T_r\}$  of $G$ satisfying
	\begin{enumerate}[(i)]
		\item $T_0=\{e\}$,
		\item for any $i\in \{1,\ldots,r\}$, there exists some $j\in \{1,\ldots,r\}$ such that $\underline{T_i}^{(-1)}=\underline{T_j}$,
		\item  for any $i,j\in \{1,\ldots,r\}$, there exist integers $p_{i,j}^k$ ($0\leq k\leq r$) such that $$\underline{T_i}\cdot \underline{T_j}=\sum_{k=0}^rp_{ij}^k \cdot \underline{T_k},$$
	\end{enumerate}
	then the $\mathbb{Z}$-module $\mathcal{S}$ spanned by $\underline{T_0},\underline{T_1},\ldots,\underline{T_r}$ is exactly a subalgebra of $\mathbb{Z}G$, and called a \textit{Schur ring} over $G$. In this situation, the  basis $\{\underline{T_0},\underline{T_1},\ldots,\underline{T_r}\}$ is called the \textit{simple basis} of the Schur ring $\mathcal{S}$. We say that the Schur ring  $\mathcal{S}$ \textit{primitive} if $\langle T_i\rangle=G$ for every $i\in \{1,\ldots,r\}$. In partitular, if $T_0=\{e\}$ and $T_1=G\setminus \{e\}$, then the Schur ring spanned  by $\underline{T_0}$ and $\underline{T_1}$ is called \textit{trivial}. Note that a trivial Schur ring is primitive. 
	
	\begin{lemma}(\cite[(3.5)]{PG})\label{lem::automorphism of distance module}
		Let $G$ denote an abelian group of order $n$, and let $m$ denote an integer coprime to $n$. Suppose that $\mathcal{S}\subseteq \mathbb{Z}H$ is a Schur ring over $G$ with simple basis $\mathcal{B}=\{\underline{T_0},\underline{T_1},\ldots,\underline{T_r}\}$. Then  $\mathcal{B}=\{\underline{T_0}^{(m)},\underline{T_1}^{(m)},\ldots,\underline{T_r}^{(m)}\}$. Moreover, the mapping $\underline{T}\rightarrow \underline{T}^{(m)}$ is an automorphism of the $\mathbb{Z}$-algebra $\mathcal{S}$. 
	\end{lemma}
	
	Let $\Gamma=\mathrm{Cay}(G,S)$ be a Cayley graph of diameter $d$. Denote by
	$$
	\mathcal{N}_i=\{g\in G\mid \partial_\Gamma(g,1)=i\}.
	$$
	The $\mathbb{Z}$-submodule of  $\mathbb{Z}G$ spanned by  $\underline{\mathcal{N}_0},\underline{\mathcal{N}_1},\ldots,\underline{\mathcal{N}_d}$ is called the \textit{distance module} of  $\Gamma$, and is denoted by \textit{$\mathcal{D}_\mathbb{Z}(G,S)$}. In \cite{MP03},  Miklavi\v{c} and Poto\v{c}nik  provided an algebraic characterization for  (primitive) distance-regular Cayley graphs in terms of the Schur ring and the distance module.
	
	\begin{lemma}(\cite[Proposition 3.6]{MP03}) \label{lem::Schur_DRG}
		Let $\Gamma=\mathrm{Cay}(G,S)$  denote a distance-regular Cayley graph and $\mathcal{D}=\mathcal{D}_\mathbb{Z}(G,S)$ its distance module. Then:
		\begin{enumerate}[$(i)$]\setlength{\itemsep}{0pt}
			\item $\mathcal{D}$ is a (primitive) Schur ring over $G$ if and only if $\Gamma$ is a (primitive) distance-regular graph;
			\item $\mathcal{D}$ is the trivial Schur ring  over $G$ if and only if $\Gamma$ is isomorphic to the complete graph.
		\end{enumerate}
	\end{lemma}

	\begin{lemma}(\cite[Kochendorfer's theorem]{KT})\label{lem::Schur_dq}
		Let $p$ be a prime, and let $a,b$ be positive integers with $a>b$. Then there is no non-trivial primitive Schur ring over the group $\mathbb{Z}_{p^{a}}\oplus\mathbb{Z}_{p^{b}}$.
	\end{lemma}
	
	If $\Gamma$ is a primitive distance-regular Cayley graph over $\mathbb{Z}_{p^s}\oplus\mathbb{Z}_{p}$  with $s\geq2$, then its distance module would be a primitive Schur ring over $\mathbb{Z}_{p^s}\oplus\mathbb{Z}_{p}$ by Lemma \ref{lem::Schur_DRG} (i), and hence must be  the trivial Schur ring  by Lemma \ref{lem::Schur_dq}. Therefore, by Lemma \ref{lem::Schur_DRG} (ii), we obtain the following result.
	
	\begin{corollary}\label{cor::pri_DRG}
		Let $\Gamma$ be a primitive distance-regular Cayley graph over $\mathbb{Z}_{p^s}\oplus\mathbb{Z}_{p}$ with $s\geq2$. Then $\Gamma$ is isomorphic to the complete graph $K_{p^{s+1}}$.
	\end{corollary}

	Let $G$ be a transitive permutation group acting on a set $\Omega$. An \textit{imprimitivity system}  for $G$ is a partition $\mathcal{B}$ of $\Omega$ which is invariant under the action of $G$, i.e., for every block $B\in \mathcal{B}$ and for every $g\in G$, we have either $B^g=B$ or $B^g\cap B=\emptyset$.
	
	\begin{lemma}(\cite[Lemma 2.2]{MP07}) \label{lem::block}
		Let $\Gamma=\mathrm{Cay}(G,S)$ denote a Cayley graph with the group $G$ acting regularly on the vertex set of $\Gamma$ by left multiplication. Suppose there exists an imprimitivity system $\mathcal{B}$ for $G$. Then the block $B\in\mathcal{B}$ containing the identity  $1\in G$ is a subgroup in $G$. Moreover,
		\begin{enumerate}[$(i)$]\setlength{\itemsep}{0pt}
			\item if $B$ is normal in $G$, then $\Gamma_\mathcal{B}=\mathrm{Cay}(G/B,S/B)$, where $S/B=\{sB\mid s\in S\setminus B\}$;
			\item if there exists an abelian subgroup $A$ in $G$ such that $G = AB$, then $\Gamma_\mathcal{B}$ is isomorphic to a
			Cayley graph on the group $A/(A\cap B)$.
		\end{enumerate}
	\end{lemma}
	
	By Lemmas \ref{lem::imprimitive} and \ref{lem::block}, we obtain the following corollary.
	
	\begin{corollary}\label{cor::DRG_dq} 
		Let $\Gamma$ denote a distance-regular Cayley graph over $\mathbb{Z}_{p^s}\oplus\mathbb{Z}_{p}$ with $s\geq 2$.
		If $\Gamma$ is antipodal, then its antipodal quotient $\overline{\Gamma}$ is a distance-regular circulant or a distance-regular Cayley graph over $\mathbb{Z}_{p^l}\oplus\mathbb{Z}_{p}$ with  $l\in\{1,2,\ldots,s-1\}$.
		
	\end{corollary}
	\begin{proof}
		Assume that $\Gamma$ is antipodal. Since $\mathbb{Z}_{p^s}\oplus\mathbb{Z}_{p}$ acts regularly on the vertex set of $\Gamma$ by left multiplication,  the antipodal classes of $\Gamma$ form an imprimitivity system $\mathcal{B}$ for $\mathbb{Z}_{p^s}\oplus\mathbb{Z}_{p}$. Let $B\in \mathcal{B}$ denote the antipodal class of $\Gamma$ containing the identity of $\mathbb{Z}_{p^s}\oplus\mathbb{Z}_{p}$.  By Lemma \ref{lem::block}, $B$ is a subgroup of $\mathbb{Z}_{p^s}\oplus\mathbb{Z}_{p}$. Clearly, $B$ is a normal subgroup, and it follows from  Lemma \ref{lem::block} (i) that $\overline{\Gamma}=\Gamma_\mathcal{B}=\mathrm{Cay}(G/B,S/B)$, which is a  circulant or a Cayley graph over $\mathbb{Z}_{p^l}\oplus\mathbb{Z}_{p}$ with $l\in\{1,2,\ldots,s-1\}$. Note that $\overline{\Gamma}$ is distance-regular by Lemma \ref{lem::imprimitive}. Thus the result follows.
	\end{proof}

	A \textit{transversal design} $TD(r,v)$ of order $v$ with line size $r$ ($r\leq v$) is a triple $(\mathcal{P},\mathcal{G},\mathcal{L})$ such that 
	\begin{enumerate}[(i)]
		\item $\mathcal{P}$ is a set of $rv$ elements (called \textit{points});
		\item $\mathcal{G}$ is a partition of $\mathcal{P}$ into $r$ classes, each of size $v$ (called \textit{groups}); 
		\item  $\mathcal{L}$ is a collection of $r$-subsets of $\mathcal{P}$ (called \textit{lines});
		\item two distinct points are contained in  a unique line if and only if they are in distinct groups.
	\end{enumerate}
	It follows immediately that  $|G\cap L|=1$ for every $G\in \mathcal{G}$ and every $L\in \mathcal{L}$, and $|\mathcal{L}|=v^2$. The \textit{line graph} of a transversal design $TD(r,v)$ is the graph with lines as vertices and two of them being adjacent whenever there is a point incident to both lines. It is known that the line graph of a transversal design $TD(r,v)$ is a strongly regular  graph with parameters $(v^2,r(v-1),v+r^2-3r,r^2-r)$ (cf. \cite[p. 122]{Jur95}). 
	For $r=2$ we get the lattice graph $K_v\times K_v$, which is the Cartesian product of two copies of $K_v$. For $r=v$ we get the complete multipartite graph $K_{v\times v}$.
	
	\begin{example}\label{examp::PCP}
		\rm Let $G$ be a group of order $v^2$. A \textit{partial congruence partition} of $G$ with degree $r$ (an $(v,r)$-PCP) is a set $\mathcal{H}$ of $r$ subgroups  of order $n$ in $G$ such that $H_1\cap H_2=\{1\}$ for every pair of distinct elements $H_1, H_2$ of $\mathcal{H}$ (cf. \cite[pp. 223--224]{M94}). For every $g\in G$ and every $H\in \mathcal{H}$, we define $\overline{g}=\{gH: H\in \mathcal{H}\}$ and $\overline{H}=\{gH: g\in G\}$. Let $\mathcal{P}=\{gH: g\in G, H\in\mathcal{H}\}$, $\mathcal{G}=\{\overline{H}: H\in \mathcal{H}\}$ and $\mathcal{L}=\{\overline{g}:g\in G\}$. Suppose that $2\leq r\leq v$. Then one can verify that $(\mathcal{P},\mathcal{G},\mathcal{L})$ is a transversal design $TD(r,v)$, and its line graph is isomorphic to  $\mathrm{Cay}(G,\cup_{H\in \mathcal{H}}H\setminus\{1\})$. 
	\end{example}

	\begin{lemma}(\cite[Proposition 6.1.4]{Jur98})\label{lem::antipodal cover of line graph}
		An antipodal cover of the line graph $G$ of a transversal design $TD(r,v)$, $r\leq v$, has diameter four when $r=2$ and diameter three otherwise. 
	\end{lemma}
	
	The \textit{Hamming graph} $H(n,q)$ has as vertex set the collection of all $n$-tuples with entries in a fixed set of size $q$, where two $n$-tuples are adjacent when they differ in only one coordinate. Note that $H(2,v)$ is just the lattice graph $K_v\times K_v$.
	
	\begin{lemma}(\cite[Proposition 5.1]{HG})\label{lem::anti_Ham}
		Let $n,q\geq2$. Then $H(n,q)$ has no distance-regular antipodal covers, except for $H(2,2)$.
	\end{lemma}

	Let $n$ be a positive integer, and let $\omega$ be a primitive $n$-th root of unity. Let  $\mathbb{F}=\mathbb{Q}(\omega)$ denote the $n$-th cyclotomic field over the rationals. For a subset $A\subseteq \mathbb{Z}_n$, let $\Delta_A:\mathbb{Z}_n\rightarrow \mathbb{F}$ be the \textit{characteristic function} of $A$, that is, $\Delta_A(z)=1$ if $z\in A$, and $\Delta_A(z)=0$ otherwise. In particular, if $A=\{a\}$, then we write $\Delta_a$ instead of $\Delta_{\{a\}}$.
	Let $\mathbb{F}^{\mathbb{Z}_n}$ be the $\mathbb{F}$-vector space consisting of all functions $f:\mathbb{Z}_n\rightarrow \mathbb{F}$ with the scalar multiplication and addition defined point-wise. Let  $(\mathbb{F}^{\mathbb{Z}_n},\cdot)$  denote the  $\mathbb{F}$-algebra obtained from $\mathbb{F}^{\mathbb{Z}_n}$ by defining the multiplication point-wise, and let $(\mathbb{F}^{\mathbb{Z}_n},\ast)$  denote the $\mathbb{F}$-algebra obtained from $\mathbb{F}^{\mathbb{Z}_n}$ by defining the multiplication as the \textit{convolution}  (see \cite{MP07}):
	\begin{equation}\label{equ::fourier}
		(f\ast g)(z)=\sum_{i\in \mathbb{Z}_n}f(i)g(z-i),~~f,g\in \mathbb{F}^{\mathbb{Z}_n}.
	\end{equation}
	The \textit{Fourier transformation} $\mathcal{F}:(\mathbb{F}^{\mathbb{Z}_n},\ast)\rightarrow(\mathbb{F}^{\mathbb{Z}_n},\cdot)$ is defined by
	\begin{equation}\label{equ::fourier1}
		(\mathcal{F}f)(z)=\sum_{i\in \mathbb{Z}_n}f(i)\omega^{iz},~~f\in \mathbb{F}^{\mathbb{Z}_n}.
	\end{equation}
	It is easy to verify that $\mathcal{F}$ is an algebra isomorphism from  $(\mathbb{F}^{\mathbb{Z}_n},\ast)$ to $(\mathbb{F}^{\mathbb{Z}_n},\cdot)$, which obeys the inversion formula \begin{equation}\label{equ::inversion formula}
		\mathcal{F}(\mathcal{F}(f))(z)=nf(-z).
	\end{equation}
	
	Let $\mathbb{Z}_n^\ast=\{i\in \mathbb{Z}_n\mid \mathrm{gcd}(i,n)=1\}$ denote the multiplicative group of units in the ring  $\mathbb{Z}_n$. Then $\mathbb{Z}_n^\ast$ acts on $\mathbb{Z}_n$ by multiplication, and each orbit of this action consists of all elements of a given order in the additive group $\mathbb{Z}_n$. Consequently, each orbit is of  the form  $O_r=\{c\cdot \frac{n}{r}\in \mathbb{Z}_n\mid c\in \mathbb{Z}_n^\ast\}$, where $r$ is some positive divisor of $n$.
	
	The following three lemmas present  basic facts about Fourier transformation.
	
	\begin{lemma}(\cite[Corollary 3.2]{MP07})\label{lem::Fourier2}
		If $A$ is a subset of $\mathbb{Z}_{p^s}$ and $\mathrm{Im}(\mathcal{F}\Delta_A)\subseteq\mathbb{Q}$, then $A$ is a union of some orbits of the action of $\mathbb{Z}_n^\ast$ on $\mathbb{Z}_n$ by multiplication, and $\mathrm{Im}(\mathcal{F}\Delta_A)\subseteq\mathbb{Z}$.
	\end{lemma}
	
	%\begin{lemma}(\cite[Lemma 3.3]{MP07})\label{lem::Fourier4}
	% Let $r$ be a  positive divisor  of $n$, and let $\omega$ be a primitive $n$-th root of unity. If $A$ is a subset of $\mathbb{Z}_{n}$, then
	%$$\mathcal{F}\Delta_{A}\left(\frac{n}{r}\right)=e_0+e_1\xi+\cdots+e_{r-1}\xi^{r-1},$$
	%where $\xi=\omega^{\frac{n}{r}}$ and $e_i =|A\cap (i+r\mathbb{Z}_n)|$ for $0\leq i\leq r-1$.
	%\end{lemma}

	Let $H$ be a subgroup of $G$. A \textit{transversal} of  $H$  is a subset of $G$ that contains exactly one element from each of the right cosets of $H$ in $G$.

	\begin{lemma}(\cite[Lemma 3.4]{MP07})\label{lem::Fourier1}
		Let $r$ be a positive divisor of $n$, and let $A$ be a transversal of the subgroup $r\mathbb{Z}_n$ in $\mathbb{Z}_n$. If $z=m\frac{n}{r}$ ($m\not\in r\mathbb{Z}_n$) is an arbitrary element of $\frac{n}{r}\mathbb{Z}_n\setminus\{0\}$, then $\mathcal{F}\Delta_A(z)=0$.
	\end{lemma}

	\begin{lemma}(\cite[Lemma 4.3]{MP07})\label{lem::Fourier3}
		Let $p$ be a  prime divisor of $n$, and let $A$ be a transversal of the subgroup $\frac{n}{p}\mathbb{Z}_n$ in $\mathbb{Z}_n$. If $A$ is a union of some orbits of the action of $\mathbb{Z}_n^\ast$ on $\mathbb{Z}_n$ by multiplication, then $p = 2$ or $A = p\mathbb{Z}_n$.
	\end{lemma}

	A graph $\Gamma$ is called \textit{integral} if all its eigenvalues are integers. Let
	$G$ be a finite group, and let $\mathcal{F}_G$ be the set of all subgroups of $G$. The \textit{Boolean algebra} 
	$\mathbb{B}(\mathcal{F}_G)$ is the set whose elements are obtained by  arbitrary finite intersections, unions,
	and complements of the elements in $\mathcal{F}_G$. The minimal non-empty elements of $\mathbb{B}(\mathcal{F}_G)$ are
	called \textit{atoms}. It is known that each element of $\mathbb{B}(\mathcal{F}_G)$ is the union of some atoms, and the atoms for $\mathbb{B}(\mathcal{F}_G)$
	are the sets $[g] = \{x\in G \mid \langle x\rangle= \langle g\rangle\}$, $g\in G$.
	
	\begin{lemma}(\cite{ICG})\label{lem::int_Cay}
		Let $G$ be an abelian group. Then $\mathrm{Cay}(G,S)$ is integral if and only if  $S\in \mathbb{B}(\mathcal{F}_G)$.
	\end{lemma}

	A \textit{$(v,k,\lambda)$-difference set} in a group $G$ of order $v$ is a $k$-subset $D$ of $G$ such that every $g\in G\setminus\{1\}$ has exactly $\lambda$ representations $g=d_1d_2^{-1}$ with $d_1,d_2\in D$. The positive integer $n:= k-\lambda$ is called the \textit{order} of the difference set $D$. If $k\notin\{|G|, |G|-1, 1, 0\}$, then $D$ is \textit{non-trivial}.
	
	\begin{lemma}(\cite[Theorem 1.3]{CDS})\label{lem::diff_set1}
		If there is a $(v,k,\lambda)$-set in a cyclic group such that  $n:=k-\lambda$ is a power of a prime $p>3$, then $v$ and $n$ are coprime.
	\end{lemma}
	
	\begin{lemma}(\cite[Result 8.1]{CDS})\label{lem::diff_set2}
		Let $G$ be an abelian group containing a $(v, k, \lambda)$-difference set with $n:=k-\lambda=3^r$, $3|v$, and assume that the Sylow $3$-subgroup of $G$ is cyclic. Then $n$ is a square, and one of the following holds:
		\begin{enumerate}[$(i)$]
			\setlength{\itemsep}{0pt}
			\item $v=(25n-9)/6$ and $3^{r-1}||\lambda$.
			\item $v=(49n-9)/12$ and $3||\lambda$.
			\item $v=(64n-9)/15$ and $3||\lambda$.
		\end{enumerate}
	\end{lemma}

	\section{Distance-regular Cayley graphs over $\mathbb{Z}_{p^s}\oplus\mathbb{Z}_{p}$}\label{section::3}
	
	The main goal of this section is to prove  Theorem \ref{thm::main}, which gives a classification of distance-regular Cayley graphs over $\mathbb{Z}_{p^s}\oplus\mathbb{Z}_{p}$. For  simplicity, we keep the following notation.
	{\flushleft\bf Notation.}
	Let $p$ be an odd prime and $s\geq 1$. For $A\subseteq \mathbb{Z}_{p^s}$ and $j\in \mathbb{Z}_{p^s}$, we denote  $j+A=\{j+i\mid i\in A\}$, $jA=\{j\cdot i\mid i\in A\}$, $-A=\{-i\mid i\in A\}$, and $(A,j)=\{(i,j)\mid i\in A\}$. 
	Suppose that $\Gamma=\Gamma(p^s\times p,R_{0},R_{1},\ldots,R_{p-1}):=\mathrm{Cay}(\mathbb{Z}_{p^s}\oplus\mathbb{Z}_{p},(R_{0},0)\cup (R_{1},1)\cup\cdots\cup (R_{p-1},p-1))$ is a distance-regular Cayley graph over $\mathbb{Z}_{p^s}\oplus\mathbb{Z}_{p}$, where  $R_{0},R_{1},\ldots,R_{p-1}$  are subsets of $\mathbb{Z}_{p^s}$ such that $0\not\in R_{0}$, $R_{0}=-R_{0}$, and $R_{i}=-R_{p-i}$ for all $i\in{\left\lbrace 1,2,\ldots,p-1\right\rbrace }$. Let $R=\cup_{i=0}^{p-1}R_i$. Denote by $k$, $\lambda$, $\mu$ and $d$ the valency, the number of common neighbors of two adjacent vertices, the number of common neighbors of two vertices at distance $2$, and the diameter of $\Gamma$, respectively. For  $j\in\{0,1,\ldots,d\}$, let $\mathcal{N}_j=N_j((0,0))$ denote the set of vertices  at distance $j$ from the identity vertex $(0,0)\in\mathbb{Z}_{p^s}\oplus\mathbb{Z}_{p}$ in $\Gamma$, and let $R_{i,j}=\{u\in \mathbb{Z}_{p^s}\mid (u,i)\in \mathcal{N}_j\}$. Clearly, $R_{0,0}=\{0\}$ and $R_{i,1}=R_{i}$.
	
	%Before giving the proof of Theorem \ref{thm::main}, we first set down a sequence of lemmas.
	
	\begin{lemma}\label{lem::1}
		Let $p$ be an odd prime, and let $\Gamma=\mathrm{Cay}(\mathbb{Z}_{p}\oplus\mathbb{Z}_{p},S)$ be a Cayley graph over $\mathbb{Z}_{p}\oplus\mathbb{Z}_{p}$. Then $\Gamma$ is  distance-regular if and only if $S=\cup_{i=1}^rH_i\setminus\{(0,0)\}$, where $2\leq r\leq p+1$, and $H_i$ ($i=1,\ldots,r$) are  subgroups of order $p$ in $\mathbb{Z}_{p}\oplus\mathbb{Z}_{p}$. In this situation, $\Gamma$ is isomorphic to the line graph of a transversal design $TD(r,p)$ when $r\leq p$, and to a complete graph when $r=p+1$. In particular, $\Gamma$ is primitive if and only if $2\leq r\leq p-1$ or $r=p+1$, and $\Gamma$ is imprimitive if and only if $r=p$, in which case  $\Gamma$ is the complete multipartite graph $K_{p\times p}$.
	\end{lemma}
	
	\begin{proof}
		Assume that $\Gamma$ is a distance-regular graph. If $\Gamma$ is integral, by Lemma \ref{lem::int_Cay}, $S$ is a union of some atoms of $\mathbb{B}(\mathcal{F}_{\mathbb{Z}_{p}\oplus\mathbb{Z}_{p}})$. Note that the atoms of $\mathbb{B}(\mathcal{F}_{\mathbb{Z}_{p}\oplus\mathbb{Z}_{p}})$ are the sets $H\setminus \{(0,0)\}$ with $H$ being a subgroup of $\mathbb{Z}_{p}\oplus\mathbb{Z}_{p}$ of order $p$. Therefore, we conlude that $S=\cup_{i=1}^rH_i\setminus\{(0,0)\}$, where $2\leq r\leq p+1$, and $H_i$ ($i=1,\ldots,r$) are  subgroups of order $p$ in $\mathbb{Z}_{p}\oplus\mathbb{Z}_{p}$.
		
		Now suppose that  $\Gamma$ is not integral. By Lemmas \ref{lem::automorphism of distance module} and \ref{lem::Schur_DRG}, $\underline{\mathcal{N}_1}^{(m)}\in \{\underline{\mathcal{N}_1},\ldots,\underline{\mathcal{N}_d}\}$ for $m\in\{1,\ldots,p-1\}$. If $\underline{\mathcal{N}_1}^{(m)}=\underline{\mathcal{N}_1}$ for all $m \in \{1,\ldots,p-1\}$, then $S$ is a union of some atoms of $B(\mathcal{F}_{\mathbb{Z}_{p}\oplus\mathbb{Z}_{p}})$, and so $\Gamma$ is integral by Lemma \ref{lem::int_Cay}, a contradiction. Thus there exists some $m_0\in \{2,\ldots,p-1\}$ such that $\underline{\mathcal{N}_1}^{(m_0)}\neq \underline{\mathcal{N}_1}$, and $\underline{\mathcal{N}_1}^{(m)}=\underline{\mathcal{N}_1}$ for all $m \in \{1,\ldots,m_0-1\}$. Then  $\underline{\mathcal{N}_1}^{(m_0)}= \underline{\mathcal{N}_2}$,  and it follows that $k_1=k_2$. By Lemma \ref{lem::k1=k2}, we have either $d=2$, or $d=3$ and $\Gamma$ is a $2$-fold antipodal cover of the complete graph. Clearly, the later case cannot occur because $\Gamma$ has the odd order $p^2$. Therefore, $\Gamma$ must be a strongly regular graph. However, since $\Gamma$ is not integral, by \cite[Theorem 2.1, Theorem 2.2]{LM05} (see also \cite{LM95}), we see that the order of $\Gamma$ must be of the form $q^{2\ell+1}$ for some prime $q\equiv 1\pmod 4$, which is impossible.  
		
		Conversely, suppose that $S=\cup_{i=1}^rH_i\setminus\{(0,0)\}$, where $2\leq r\leq p+1$, and $H_i$ ($i=1,\ldots,r$) are  subgroups of order $p$ in $\mathbb{Z}_{p}\oplus\mathbb{Z}_{p}$. If $r=p+1$, then $S=\mathbb{Z}_{p}\oplus\mathbb{Z}_{p}\setminus{(0,0)}$, and $\Gamma$ is isomorphic to the complete graph $K_{v^2}$, as desired. If $r\leq p$, according the arguments in Example \ref{examp::PCP}, $\Gamma$ is isomorphic to the line graph of a transversal design $TD(r,p)$, and so is a strongly regular graph with parameters $(p^2,r(p-1),p+r^2-3r,r^2-r)$, as required. In particular, if $r=p$ then $\Gamma$ is the complete multipartite graph $K_{p\times p}$, which is imprimitive. If $2\leq r\leq p-1$, then $\Gamma$ is neither antipodal nor bipartite, and so must be  primitive.
	\end{proof}
	
	\begin{lemma}\label{lem::neighbor}
		Let  $\Gamma=\Gamma(p^s\times p,R_{0},R_{1},\ldots,R_{p-1})$ be a connected Cayley graph over $\mathbb{Z}_{p^s}\oplus\mathbb{Z}_{p}$. For every $(i,j)\in \mathbb{Z}_{p^s}\oplus \mathbb{Z}_p$, $$
		N((i,j))=(i+R_{0},j)\cup (i+R_1,j+1) \cup \cdots\cup (i+R_{p-1},j+p-1).
		$$
	\end{lemma}
	\begin{proof}
		By definition.
	\end{proof}

	\begin{lemma}\label{lem::common_neighbor}
		Let  $\Gamma=\Gamma(p^s\times p,R_{0},R_{1},\ldots,R_{p-1})$ be a connected Cayley graph over $\mathbb{Z}_{p^s}\oplus\mathbb{Z}_{p}$. For every $(i,j)\in \mathbb{Z}_{p^s}\oplus \mathbb{Z}_p$, 
		$$
		\begin{aligned}
			|N((0,0))\cap N((i,j))|&=|R_0\cap (i-R_j)|+|R_1\cap (i-R_{j-1})|+\cdots+|R_j\cap (i-R_{0})|\\
			&~~~+|R_{j+1}\cap (i-R_{p-1})|+\cdots+ |R_{p-1}\cap (i-R_{j+1})|.
		\end{aligned}
		$$
	\end{lemma}
	\begin{proof}
		The result follows from Lemma \ref{lem::neighbor} immediately.
	\end{proof}

	Let $\omega=e^{2\pi \mathbf{i}/p^s}$ be the primitive $p^s$-th root of unity, and let  $\mathbb{F}=\mathbb{Q}(\omega)$. Suppose that  $(\mathbb{F}^{\mathbb{Z}_{p^s}},\cdot)$ and $(\mathbb{F}^{\mathbb{Z}_{p^s}},\ast)$ are $\mathbb{F}$-algebras defined as in Section \ref{section::2}, and that $\mathcal{F}$ is the Fourier transformation  from $(\mathbb{F}^{\mathbb{Z}_{p^s}},\ast)$ to $(\mathbb{F}^{\mathbb{Z}_{p^s}},\cdot)$ give in  \eqref{equ::fourier1}. We denote
	\begin{equation}\label{equ::fourier5}
		\underline{\mathbf{r}}_{i,j}(z)=(\mathcal{F}\Delta_{R_{i,j}})(z)=\sum_{t\in R_{i,j}}\omega^{tz}~~\mbox,
	\end{equation}
	where $\Delta_{R_{i,j}}$ is the characteristic function of $R_{i,j}$. In particular, we denote $\underline{\mathbf{r}}_{i}=\underline{\mathbf{r}}_{i,1}=\mathcal{F}\Delta_{R_i}$. Let $\ast$ be the convolution of   $(\mathbb{F}^{\mathbb{Z}_{p^s}},\ast)$ defined  as in \eqref{equ::fourier}. For  $A,B\subseteq \mathbb{Z}_{p^s}$, we can verify that
	\begin{equation}\label{equ::fourier2}
		(\Delta_A\ast \Delta_B)(i)=|(i-A)\cap B|=|(i-B)\cap A|,~~i\in\mathbb{Z}_{p^s}.
	\end{equation}
	
	\begin{lemma}\label{lem::Fourier}
		Let  $\Gamma=\Gamma(p^s\times p,R_{0},R_{1},\ldots,R_{p-1})$ be a distance-regular Cayley graph over $\mathbb{Z}_{p^s}\oplus\mathbb{Z}_{p}$. Then
		$$
		\underline{\mathbf{r}}_{0}\underline{\mathbf{r}}_{0}+\underline{\mathbf{r}}_{1}\underline{\mathbf{r}}_{p-1}+\cdots+\underline{\mathbf{r}}_{p-1}\underline{\mathbf{r}}_{1}=k+\lambda \underline{\mathbf{r}}_{0}+\mu\underline{\mathbf{r}}_{0,2},
		$$
		and  
		$$
		\underline{\mathbf{r}}_{0}\underline{\mathbf{r}}_{j}+\underline{\mathbf{r}}_{1}\underline{\mathbf{r}}_{j-1}+\cdots+\underline{\mathbf{r}}_{j}\underline{\mathbf{r}}_{0}+\underline{\mathbf{r}}_{j+1}\underline{\mathbf{r}}_{p-1}+\cdots+\underline{\mathbf{r}}_{p-1}\underline{\mathbf{r}}_{j+1}=\lambda\underline{\mathbf{r}}_{j}+\mu \underline{\mathbf{r}}_{j,2}
		$$ 
		for $1\leq j\leq p-1$.
	\end{lemma}
	
	\begin{proof}
		Let $i\in \mathbb{Z}_{p^s}$. According to Lemma \ref{lem::common_neighbor} and \eqref{equ::fourier2},  we have
		\begin{equation}\label{equ::fourier3}
			\begin{aligned}
				&~~~~(\Delta_{R_{0}}\ast \Delta_{R_{0}})(i)+(\Delta_{R_{1}}\ast \Delta_{R_{p-1}})(i)+\cdots+(\Delta_{R_{p-1}}\ast \Delta_{R_{1}})(i)\\
				&=|R_0\cap (i-R_0)|+|R_1\cap (i-R_{p-1})|+\cdots+|R_{p-1}\cap (i-R_1)|\\
				&=|N((0,0))\cap N((i,0))|\\
				&=(k\Delta_0+\lambda\Delta_{R_0}+\mu\Delta_{R_{0,2}})(i),\\
			\end{aligned}
		\end{equation}
		and
		\begin{equation}\label{equ::fourier4}
			\begin{aligned}
				&~~~~(\Delta_{R_{0}}\ast \Delta_{R_{j}})(i)+(\Delta_{R_{1}}\ast \Delta_{R_{j-1}})(i)+\cdots+(\Delta_{R_{j}}\ast \Delta_{R_{0}})(i)+(\Delta_{R_{j+1}}\ast \Delta_{R_{p-1}})(i)\\
				&~~~~+\cdots+(\Delta_{R_{p-1}}\ast \Delta_{R_{j+1}})(i)\\
				&=|R_0\cap (i-R_j)|+|R_1\cap (i-R_{j-1})|+\cdots+|R_j\cap (i-R_{0})|+|R_{j+1}\cap (i-R_{p-1})|\\
				&~~~+\cdots+ |R_{p-1}\cap (i-R_{j+1})|\\
				&=|N((0,0))\cap N((i,j))|\\
				&=(\lambda\Delta_{R_j}+\mu\Delta_{R_{j,2}})(i)\\
			\end{aligned}
		\end{equation}
		for $1\leq j\leq p-1$. Then, by applying the Fourier transformation $\mathcal{F}$ on both sides of \eqref{equ::fourier3} and \eqref{equ::fourier4}, we obtain the result immediately.
	\end{proof}

	\begin{lemma}\label{lem::key1}
		Let $p>2$ be a prime. There are no antipodal non-bipartite distance-regular Cayley graphs over $\mathbb{Z}_{p^s}\oplus\mathbb{Z}_{p}$ ($s\geq 1$) with diameter $3$.
	\end{lemma}
	\begin{proof}
		By Lemma \ref{lem::1}, the case $s=1$ is obvious, and we can suppose that $s\geq 2$.  By contradiction, assume that $\Gamma=\Gamma(p^s\times p,R_{0},R_{1},\ldots,R_{p-1})$  is an antipodal non-bipartite distance-regular Cayley graph of diameter $3$ over $\mathbb{Z}_{p^s}\oplus\mathbb{Z}_{p}$  with $s$ ($s\geq 2$) as small as possible. Let $k$ and $r$ ($r\geq 2$) denote the valency and the common size of antipodal classes (or fibres) of $\Gamma$, respectively. According to Lemma \ref{lem::antipodal_DRG},  $k+1=\frac{p^{s+1}}{r}$,  and $\Gamma$ has the intersection array
		\begin{equation}\label{equ::1.2}
			\{k,\mu(r-1),1;1,\mu,k\}
		\end{equation}
		and eigenvalues $k$, $\theta_1$, $\theta_2=-1$, $\theta_3$, where
		\begin{equation}\label{equ::1.3}
			\theta_1=\frac{\lambda-\mu}{2}+\delta,~~\theta_3=\frac{\lambda-\mu}{2}-\delta~~\mbox{and}~~\delta=\sqrt{k+\left(\frac{\lambda-\mu}{2}\right)^2}.
		\end{equation}
		Note that $r$ is a prime power of $p$. Let $H=\mathcal{N}_3\cup\{(0,0)\}$. Then $H$ is an antipodal class of $\Gamma$, and  $|H|=r$. Since $\mathbb{Z}_{p^s}\oplus\mathbb{Z}_{p}$ acts regularly on  $V(\Gamma)$ by left multiplication, the antipodal classes of $\Gamma$ form an imprimtivity system for  $\mathbb{Z}_{p^s}\oplus\mathbb{Z}_{p}$. By Lemma \ref{lem::block}, $H$ is a subgroup of $\mathbb{Z}_{p^s}\oplus\mathbb{Z}_{p}$. If $r$ is not  prime, i.e., $r\neq p$, then $H$ has  a non-trivial subgroup $K$. Let $\mathcal{B}$ denote the partition consisting of all orbits of $K$ acting on $V(\Gamma)$ by left multiplication.  As $K$ is normal in $\mathbb{Z}_{p^s}\oplus\mathbb{Z}_{p}$, the partition $\mathcal{B}$ is also an imprimtivity system for $\mathbb{Z}_{p^s}\oplus\mathbb{Z}_{p}$. Then it follows from  Lemma \ref{lem::block} (i) that the quotient graph $\Gamma_\mathcal{B}$ is a Cayley graph over the cyclic group or the group $\mathbb{Z}_{p^{s'}}\oplus\mathbb{Z}_{p}$, where $1\leq s'\leq s-1$. Observe that $\mathcal{B}$ is  an equitable partition of $\Gamma$, and each block of $\mathcal{B}$ is contained in some fibre of $\Gamma$ and is neither a single vertex nor a fibre.  By Lemma \ref{lem::quotient_graph}, $\Gamma_\mathcal{B}$ is an antipodal distance-regular graph with diameter $3$. If $\Gamma_\mathcal{B}$ is bipartite, then $\Gamma$ is also bipartite, a contradiction. Hence, $\Gamma_\mathcal{B}$ is an antipodal non-bipartite distance-regular Cayley graph of diameter $3$ over the cyclic group or the group  $\mathbb{Z}_{p^{s'}}\oplus\mathbb{Z}_{p}$. By Lemma \ref{lem::cir_DRG}, we assert that the former case cannot occur. For the later case, 
		from Lemma \ref{lem::1} we can deduce that $s'\geq 2$. However, this violates the  minimality  of $s$.
		Therefore, $r=p$ and $k=p^s-1$. If $\lambda=\mu$, by  \eqref{equ::1.2}, we have $p^s-1=k=\mu p+1$, implying that $p=2$, a contradiction. Thus $\lambda\neq\mu$, and by Lemma \ref{lem::antipodal_DRG}, $\Gamma$ is integral, and $2\delta\in\mathbb{Z}$.  We consider the following two cases.
		
		{\flushleft \bf Case A.} $\mathcal{N}_3\not\subseteq\langle (1,0)\rangle$. 
		
		In this case, there exists an element $(i,j)\subseteq\mathcal{N}_3$ with $j\neq0$. Since $H=\mathcal{N}_3\cup \{(0,0)\}$ has the prime order $p$, there exists some element $(u,1)\in H$ such that $H=\langle (u,1)\rangle$ and $p^s|up$. Note that the mapping $\sigma_u$ defined by $\sigma_u((i,j))=(i-uj,j)$ is an automorphism of $\mathbb{Z}_{p^s}\oplus\mathbb{Z}_{p}$. As $\sigma_u((u,1))=(0,1)$, we may assume without loss of generality that $H=\langle (0,1)\rangle$.
		
		Since $\Gamma$ is antipodal, the vertices in $\mathcal{N}_3=\langle (0,1)\rangle\setminus\{(0,0)\}$ have disjoint neighborhoods, and $\mathcal{N}_2$ is just the  union of these neighborhoods.
		Recall that  $N((0,i))=(R_0,i)\cup (R_1,i+1)\cup \cdots\cup (R_{p-1},i+p-1)$. As $N((0,i))\cap N((0,i'))=\emptyset$ whenever  $i\neq i'$, we assert that  $R_j\cap R_{j'}=\emptyset$ where $0\leq j\neq j'\leq p-1$. Let $R=\cup_{i=0}^{p-1} R_i$. It is easy to see that  $R=\mathbb{Z}_{p^s}\setminus\{0\}$, and  $R_{i,2}=R\backslash R_i$ for $0\leq i\leq p-1$. Then
		\begin{equation}\label{equ::k1} \underline{\mathbf{r}}_{0}+\underline{\mathbf{r}}_{1}+\underline{\mathbf{r}}_{2}+\cdots+\underline{\mathbf{r}}_{p-1}=p^s\Delta_0-1,
		\end{equation}
		and hence
		\begin{equation*}\label{equ::k1_1} \underline{\mathbf{r}}_{i,2}=p^s\Delta_0-1-\underline{\mathbf{r}}_{i},~0\leq i\leq p-1.
		\end{equation*}
		Combining this with Lemma \ref{lem::Fourier}, we have
		\begin{equation}\label{equ::k2}
			\left\{\begin{aligned}
				&\underline{\mathbf{r}}_{0}\underline{\mathbf{r}}_{0}+\underline{\mathbf{r}}_{1}\underline{\mathbf{r}}_{p-1}+\underline{\mathbf{r}}_{2}\underline{\mathbf{r}}_{p-2}+\cdots+\underline{\mathbf{r}}_{p-1}\underline{\mathbf{r}}_{1}=k+\lambda \underline{\mathbf{r}}_{0}+\mu(p^s\Delta_0-1-\underline{\mathbf{r}}_{0}),\\ &\underline{\mathbf{r}}_{0}\underline{\mathbf{r}}_{1}+\underline{\mathbf{r}}_{1}\underline{\mathbf{r}}_{0}+\underline{\mathbf{r}}_{2}\underline{\mathbf{r}}_{p-1}+\cdots+\underline{\mathbf{r}}_{p-1}\underline{\mathbf{r}}_{2}=\lambda \underline{\mathbf{r}}_{1}+\mu(p^s\Delta_0-1-\underline{\mathbf{r}}_{1}),\\ 
				&~~~~~~~~~~~~~~~~~~~~~~~~~~~~~~~~~~\vdots\\ &\underline{\mathbf{r}}_{0}\underline{\mathbf{r}}_{p-1}+\underline{\mathbf{r}}_{1}\underline{\mathbf{r}}_{p-2}+\underline{\mathbf{r}}_{2}\underline{\mathbf{r}}_{p-3}+\cdots+\underline{\mathbf{r}}_{p-1}\underline{\mathbf{r}}_{0}=\lambda \underline{\mathbf{r}}_{p-1}+\mu(p^s\Delta_0-1-\underline{\mathbf{r}}_{p-1}).\\ \end{aligned}\right.
		\end{equation}	
		Let $\epsilon$ be a primitive $p$-th root of unity. Multiplying both sides of the equations in \eqref{equ::k2} by coefficients $1,\epsilon,\epsilon^2,\ldots,\epsilon^{p-1}$  respectively, and adding them up, we obtain 
		\begin{equation*}\label{equ::k3}
			(\underline{\mathbf{r}}_{0}+\epsilon\underline{\mathbf{r}}_{1}+\epsilon^2\underline{\mathbf{r}}_{2}+\ldots+\epsilon^{p-1}\underline{\mathbf{r}}_{p-1})^2=k+(\lambda-\mu)(\underline{\mathbf{r}}_{0}+\epsilon\underline{\mathbf{r}}_{1}+\epsilon^2\underline{\mathbf{r}}_{2}+\ldots+\epsilon^{p-1}\underline{\mathbf{r}}_{p-1}) .
		\end{equation*}	
		Thus $\mathrm{Im}(\underline{\mathbf{r}}_{0}+\epsilon\underline{\mathbf{r}}_{1}+\epsilon^2\underline{\mathbf{r}}_{2}+\cdots+\epsilon^{p-1}\underline{\mathbf{r}}_{p-1})\subseteq\{\theta_1,\theta_3\}$, and we can assume that
		\begin{equation}\label{equ::k4}
			\underline{\mathbf{r}}_{0}+\epsilon\underline{\mathbf{r}}_{1}+\epsilon^2\underline{\mathbf{r}}_{2}+\cdots+\epsilon^{p-1}\underline{\mathbf{r}}_{p-1}=\theta_1\Delta_{A}+\theta_3\Delta_{B},
		\end{equation}	
		where $\{A,B\}$ is a partition of $\mathbb{Z}_{p^s}$.
		Combining \eqref{equ::1.3} and \eqref{equ::k4} yields that
		\begin{equation}\label{equ::a1}
			\Delta_A=\frac{1}{2\delta}(\underline{\mathbf{r}}_{0}+\epsilon\underline{\mathbf{r}}_{1}+\epsilon^2\underline{\mathbf{r}}_{2}+\cdots+\epsilon^{p-1}\underline{\mathbf{r}}_{p-1}-\theta_3).
		\end{equation}
		As $\underline{\mathbf{r}}_0(-z)=\underline{\mathbf{r}}_{0}(z)$ and $\underline{\mathbf{r}}_i(-z)=\underline{\mathbf{r}}_{p-i}(z)$ for $1\leq i\leq p-1$, we can deduce from \eqref{equ::a1} that
		\begin{equation}\label{equ::a1_1}
			\Delta_{-A}=\frac{1}{2\delta}(\underline{\mathbf{r}}_{0}+\epsilon\underline{\mathbf{r}}_{p-1}+\epsilon^2\underline{\mathbf{r}}_{p-2}+\cdots+\epsilon^{p-1}\underline{\mathbf{r}}_{1}-\theta_3).
		\end{equation}
		By applying the Fourier transformation $\mathcal{F}$ to both sides of \eqref{equ::a1} and \eqref{equ::a1_1}, and then  using \eqref{equ::inversion formula}, we obtain
		\begin{equation}\label{equ::a2}
			\mathcal{F}\Delta_A=\frac{p^s}{2\delta}(\Delta_{R_{0}}+\epsilon\Delta_{R_{p-1}}+\epsilon^2\Delta_{R_{p-2}}+\cdots+\epsilon^{p-1}\Delta_{R_{1}}-\theta_3\Delta_0),
		\end{equation}
		and
		\begin{equation}\label{equ::a2_1}
			\mathcal{F}\Delta_{-A}=\frac{p^s}{2\delta}(\Delta_{R_{0}}+\epsilon\Delta_{R_{1}}+\epsilon^2\Delta_{R_{2}}+\cdots+\epsilon^{p-1}\Delta_{R_{p-1}}-\theta_3\Delta_0).
		\end{equation}
		Combining  \eqref{equ::inversion formula},  \eqref{equ::fourier2}, \eqref{equ::k1}, \eqref{equ::a2} and \eqref{equ::a2_1}, we can deduce that 
		\begin{equation*}\label{equ::a3}
			\begin{aligned}
				|A\cap (i+A)|&=(\Delta_{A}\ast\Delta_{-A})(i)\\
				&=\frac{1}{p^s}\cdot \mathcal{F}((\mathcal{F}\Delta_{A})(\mathcal{F}\Delta_{-A}))(i)\\
				&=\frac{p^s}{4\delta^2}\cdot \mathcal{F}(\Delta_{R_{0}}+\Delta_{R_{1}}+\Delta_{R_{2}}+\cdots+\Delta_{R_{p-1}}+\theta_3^2\Delta_0)(i)\\
				&=\frac{p^s}{4\delta^2}\cdot (\underline{\mathbf{r}}_{0}+\underline{\mathbf{r}}_{1}+\underline{\mathbf{r}}_{2}+\cdots+\underline{\mathbf{r}}_{p-1}+\theta_3^2)(i)\\
				&=\frac{p^s}{4\delta^2}\cdot (p^s\Delta_0-1+\theta_3^2)(i)
			\end{aligned}
		\end{equation*}
		for any $i\in\mathbb{Z}_{p^s}$. 
		Thus we conclude that $A$ is a  $(p^s,\frac{p^{2s}}{4\delta^2}+\frac{p^s}{4\delta^2}(\theta_3^2-1),\frac{p^s}{4\delta^2}(\theta_3^2-1))$-difference set of $\mathbb{Z}_{p^s}$, and it follows that $\frac{p^{2s}}{4\delta^2}\in\mathbb{Z}$. Since $2\delta\in\mathbb{Z}$, we assert that  $\frac{p^{2s}}{4\delta^2}=(\frac{p^s}{2\delta})^2$ must be a power of prime $p$, and so is $\frac{p^s}{2\delta}$.  If $p>3$, 
		by Lemma \ref{lem::diff_set1}, $\frac{p^{2s}}{4\delta^2}$ is coprime to $p^s$, and so $\frac{p^{2s}}{4\delta^2}=1$, i.e.,  $k+1=p^s=2\delta$. Then from  \eqref{equ::1.3} we can deduce that  $(k+1)^2=4k+(\lambda-\mu)^2$, and hence $k-1=\lambda-\mu$ or $\mu-\lambda$. Since $k=\mu(p-1)+\lambda+1>\mu+\lambda+1$, we obtain a contradiction immediately. Now suppose that $p=3$. As above, the case  $p^s=2\delta$ cannot occur, and we can suppose that  $\frac{p^s}{2\delta}=\frac{3^s}{2\delta}=3^t$ for some $1\leq t\leq s$. Then, by Lemma \ref{lem::diff_set2}, one of the following holds:
		\begin{enumerate}[$(i)$]\setlength{\itemsep}{0pt}
			\item $6\times 3^s=(5\times 3^t+3)(5\times 3^t-3)$;
			\item $12\times 3^s=(7\times 3^t+3)(7\times 3^t-3)$;
			\item $15\times 3^s=(8\times 3^t+3)(8\times 3^t-3)$.
		\end{enumerate}
		Note that the difference of the two factors on the right side of (i)-(iii) is always equal to $6$. A careful calculation shows that the left side of (i) and (ii) cannot be written as the product of two factors with  difference  $6$, while the left side of (iii) can be written as the the product of two factors with difference $6$ only when $s=2$. However, in this situation, the equality in (iii) cannot hold because $(8\times 3^t+3)(8\times 3^t-3)\geq (8\times 3+3)(8\times 3-3)>15\times 3^2$.
		Therefore, we obtain a contradiction. 
		
		{\flushleft \bf Case B.} $\mathcal{N}_3\subseteq\langle (1,0)\rangle$.  
		
		In this case, $H=\mathcal{N}_3\cup \{(0,0)\}$ is the subgroup of $\langle (1,0)\rangle$ with order $p$, and therefore  $\mathcal{N}_3=\{(ip^{s-1},0)\mid i=1,2,\ldots,p-1\}$. Hence,  $R_{0,3}= p^{s-1}\mathbb{Z}_{p^s}\setminus\{0\}$, and $R_{i,3}=\emptyset$ whenever $i\neq0$.  Before going further, similarly as  in \cite[Lemma 4.4]{MP07}, we need the following claim.

		\begin{claim}\label{claim::1}
			The sets  $R_0\cup \{0\}$ and $R_i$ where $i\neq0$ are transversals of the subgroup $p^{s-1}\mathbb{Z}_{p^s}$ in $\mathbb{Z}_{p^s}$.
		\end{claim}
		\renewcommand\proofname{\it{Proof of Claim \ref{claim::1}}}
		\begin{proof}
			First assume that $|R_i\cap (\ell+p^{s-1}\mathbb{Z}_{p^s})|\geq 2$ for some $\ell\in \mathbb{Z}_{p^s}$. Then there exists some $j\in\{1,\ldots,p-1\}$ such that $jp^{s-1}\in R_i-R_i$. Thus  $(j p^{s-1},0)\in\mathcal{N}_1\cup \mathcal{N}_2$, contrary to  the fact that $(j p^{s-1},0)\in \mathcal{N}_3$. Hence, $|R_i\cap (\ell+p^{s-1}\mathbb{Z}_{p^s})|\leq 1$ for all $\ell\in \mathbb{Z}_{p^s}$. Similarly, $|(R_0\cup \{0\})\cap (\ell+p^{s-1}\mathbb{Z}_{p^s})|\leq 1$ for all $\ell\in \mathbb{Z}_{p^s}$. Now assume that $R_i\cap (\ell+p^{s-1}\mathbb{Z}_{p^s})=\emptyset$ for some  $\ell\in \mathbb{Z}_{p^s}$. Then $\ell+p^{s-1}\mathbb{Z}_{p^s}\subseteq R_{i,2}$ due to $R_{i,0}=R_{i,3}=\emptyset$. Since each vertex of $\mathcal{N}_2$ has a neighbor in $\mathcal{N}_3$, there exists some $j\in\{1,\ldots,p-1\}$ such that $(jp^{s-1},0)\in \mathcal{N}_3$ is adjacent to $(\ell+p^{s-1},i)\in \mathcal{N}_2$. This implies that  $\ell+(1-j)p^{s-1}\in R_i$, which is impossible because $R_i\cap (\ell+p^{s-1}\mathbb{Z}_{p^s})=\emptyset$. Hence, $R_i$ has non-empty intersection with every coset of $p^{s-1}\mathbb{Z}_{p^s}$ in $\mathbb{Z}_{p^s}$. Similarly, $R_0\cup \{0\}$ has non-empty intersection with every coset of $p^{s-1}\mathbb{Z}_{p^s}$ in $\mathbb{Z}_{p^s}$. Therefore, we conclude that  $R_i$ and $R_0\cup \{0\}$ are transversals of the subgroup $p^{s-1}\mathbb{Z}_{p^s}$ in $\mathbb{Z}_{p^s}$. 
		\end{proof}
		
		By Claim \ref{claim::1}, $|R_0|=p^{s-1}-1$, and $|R_i|=p^{s-1}$ for $1\leq i\leq p-1$. Since $R_{0,2}=\mathbb{Z}_{p^s}\setminus(p^{s-1}\mathbb{Z}_{p^s}\cup R_0)$ and $R_{i,2}=\mathbb{Z}_{p^s}\setminus R_i$, we have $|R_{0,2}|=(p-1)|R_0|$ and $|R_{i,2}|=(p-1)|R_i|$  for $1\leq i\leq p-1$. Furthermore, by \eqref{equ::fourier5}, we obtain $\underline{\mathbf{r}}_{0,2}=p^s\Delta_0-p\Delta_{p\mathbb{Z}_{p^s}}-\underline{\mathbf{r}}_0$
		and $\underline{\mathbf{r}}_{i,2}=n\Delta_0-\underline{\mathbf{r}}_i$. Combining this with  Lemma \ref{lem::Fourier} yields that
		\begin{equation}\label{equ::l1}
			\left\{\begin{aligned}
				&\underline{\mathbf{r}}_{0}\underline{\mathbf{r}}_{0}+\underline{\mathbf{r}}_{1}\underline{\mathbf{r}}_{p-1}+\underline{\mathbf{r}}_{2}\underline{\mathbf{r}}_{p-2}+\cdots+\underline{\mathbf{r}}_{p-1}\underline{\mathbf{r}}_{1}=k+\lambda \underline{\mathbf{r}}_{0}+\mu(p^s\Delta_0-p\Delta_{p\mathbb{Z}_{p^s}}-\underline{\mathbf{r}}_{0}),\\ &\underline{\mathbf{r}}_{0}\underline{\mathbf{r}}_{1}+\underline{\mathbf{r}}_{1}\underline{\mathbf{r}}_{0}+\underline{\mathbf{r}}_{2}\underline{\mathbf{r}}_{p-1}+\cdots+\underline{\mathbf{r}}_{p-1}\underline{\mathbf{r}}_{2}=\lambda \underline{\mathbf{r}}_{1}+\mu(p^s\Delta_0-\underline{\mathbf{r}}_{1}),\\ 
				&~~~~~~~~~~~~~~~~~~~~~~~~~~~~~~~~~~\vdots\\ &\underline{\mathbf{r}}_{0}\underline{\mathbf{r}}_{p-1}+\underline{\mathbf{r}}_{1}\underline{\mathbf{r}}_{p-2}+\underline{\mathbf{r}}_{2}\underline{\mathbf{r}}_{p-3}+\cdots+\underline{\mathbf{r}}_{p-1}\underline{\mathbf{r}}_{0}=\lambda \underline{\mathbf{r}}_{p-1}+\mu(p^s\Delta_0-\underline{\mathbf{r}}_{p-1}).\\ \end{aligned}\right.
		\end{equation}
		Clearly, $\underline{\mathbf{r}}_0(0)=|R_0|=p^{s-1}-1$ and $\underline{\mathbf{r}}_1(0)=|R_1|=p^{s-1}$. Moreover, by Lemma \ref{lem::Fourier1} and Claim \ref{claim::1}, we have 
		\begin{equation}\label{equ::key0}
			\underline{\mathbf{r}}_0(z)=-1,~\mbox{for any}~z\in p\mathbb{Z}_{p^s}\setminus\{0\}.
		\end{equation}
		Now suppose $z\not\in p\mathbb{Z}_{p^s}$. Let $\epsilon$ be a primitive $p$-th root of unity. For any $i\in\{0,1,\ldots,p-1\}$, multiplying both sides of the equations in \eqref{equ::l1} by coefficients $1,\epsilon^i,\epsilon^{2i},\ldots,\epsilon^{(p-1)i}$  respectively, and adding them up, we obtain 
		\begin{equation*}\label{equ::key1}
			(\underline{\mathbf{r}}_{0}+\epsilon^i\underline{\mathbf{r}}_{1}+\epsilon^{2i}\underline{\mathbf{r}}_{2}+\cdots+\epsilon^{(p-1)i}\underline{\mathbf{r}}_{p-1})^2(z)=k+(\lambda-\mu)\cdot (\underline{\mathbf{r}}_{0}+\epsilon^i\underline{\mathbf{r}}_{1}+\epsilon^{2i}\underline{\mathbf{r}}_{2}+\cdots+\epsilon^{(p-1)i}\underline{\mathbf{r}}_{p-1})(z),
		\end{equation*}
		and hence 
		\begin{equation*}
			(\underline{\mathbf{r}}_{0}+\epsilon^i\underline{\mathbf{r}}_{1}+\epsilon^{2i}\underline{\mathbf{r}}_{2}+\cdots+\epsilon^{(p-1)i}\underline{\mathbf{r}}_{p-1})(z)\in\{\theta_1,\theta_3\},~z\not\in p\mathbb{Z}_{p^s}.
		\end{equation*}
		As $\Gamma$ is integral, we have 
		\begin{equation*}
			(\underline{\mathbf{r}}_{0}+\epsilon^i\underline{\mathbf{r}}_{1}+\epsilon^{2i}\underline{\mathbf{r}}_{2}+\cdots+\epsilon^{(p-1)i}\underline{\mathbf{r}}_{p-1})(z)\in\mathbb{Z},~z\not\in p\mathbb{Z}_{p^s},
		\end{equation*}
		for all $i\in\{0,1,\ldots,p-1\}$.  Therefore,
		$$
		\sum_{i=0}^{p-1}(\underline{\mathbf{r}}_{0}+\epsilon^i\underline{\mathbf{r}}_{1}+\epsilon^{2i}\underline{\mathbf{r}}_{2}+\cdots+\epsilon^{(p-1)i}\underline{\mathbf{r}}_{p-1})(z)= (p-1)\cdot\underline{\mathbf{r}}_{0}(z)\in \mathbb{Z},~z\not\in p\mathbb{Z}_{p^s}.
		$$
		Combining this with \eqref{equ::key0}, we assert that $\mathrm{Im}(\underline{\mathbf{r}}_{0})\subseteq\mathbb{Q}$. By Lemma \ref{lem::Fourier2}, $R_0$ is a union of some orbits of  $\mathbb{Z}_{p^s}^\ast$ acting on $\mathbb{Z}_{p^s}$, and so is $R_0\cup \{0\}$. Then, by Lemma \ref{lem::Fourier3} and Claim \ref{claim::1},  we obtain $R_0=p\mathbb{Z}_{p^s}\setminus\{0\}$. This implies that   $\underline{\mathbf{r}}_{0}=p^{s-1}\Delta_{p^{s-1}\mathbb{Z}_{p^s}}-1$, and hence $\underline{\mathbf{r}}_{0}(p^{s-1})=p^{s-1}-1$. On the other hand, since  $s\geq 2$, we have $p^{s-1}\in p\mathbb{Z}_{p^s}\setminus\{0\}$, and so  $\underline{\mathbf{r_0}}(p^{s-1})=-1$ by \eqref{equ::key0}, a contradiction.

		Therefore, we conclude that there are no antipodal non-bipartite distance-regular graphs over $\mathbb{Z}_{p^s}\oplus\mathbb{Z}_{p}$  with diameter $3$.
	\end{proof}

	Now we are in a position to give the proof of Theorem \ref{thm::main}.
	
	\renewcommand\proofname{\it{Proof of Theorem \ref{thm::main}}}
	\begin{proof}
		First of all, we  show that the graphs listed in (i)--(iii) are distance-regular Cayley graph over $\mathbb{Z}_{p^s}\oplus\mathbb{Z}_{p}$. For (i): the complete graph $K_{p^{s+1}}$ is distance-regular with diameter $1$, and $K_{p^{s+1}}\cong \mathrm{Cay}(\mathbb{Z}_{p^s}\oplus\mathbb{Z}_{p},(\mathbb{Z}_{p^s}\oplus\mathbb{Z}_{p})\setminus\{(0,0)\})$.  For (ii):  the complete multipartite graph $K_{t\times m}$ ($tm=p^{s+1}$)  is distance-regular with diameter $2$. Since $m$  is a divisor of  $p^{s+1}$, there exists a subgroup $H^{(m)}$ of order $m$ in $\mathbb{Z}_{p^s}\oplus\mathbb{Z}_{p}$. Then it is easy to see that  $K_{t\times m}\cong \mathrm{Cay}(\mathbb{Z}_{p^s}\oplus\mathbb{Z}_{p}, (\mathbb{Z}_{p^s}\oplus\mathbb{Z}_{p})\setminus H^{(m)})$. For (iii):  by Lemma \ref{lem::1}, the graph $\mathrm{Cay}(\mathbb{Z}_{p}\oplus\mathbb{Z}_{p},S)$ with $S=\cup_{i=1}^rH_i\setminus\{(0,0)\}$, where $2\leq r\leq p-1$ and $H_i$ ($1\leq i\leq r$) are subgroups of order $p$ in $\mathbb{Z}_{p}\oplus\mathbb{Z}_{p}$,  is a distance-regular graph of diameter $2$.
		
		Conversely, suppose that $\Gamma=\mathrm{Cay}(\mathbb{Z}_{p^s}\oplus\mathbb{Z}_{p},S)$ is a  distance-regular Cayley graph over $\mathbb{Z}_{p^s}\oplus\mathbb{Z}_{p}$. If $s=1$, then the result follows from Lemma \ref{lem::1} immediately. Thus we can suppose $s\geq 2$.  If $\Gamma$ is primitive, by Corollary \ref{cor::pri_DRG}, $\Gamma$ is isomorphic to the complete graph $K_{p^{s}+1}$. Now assume that $\Gamma$ is imprimitive. Clearly, $\Gamma\ncong C_{p^{s+1}}$ because $\mathbb{Z}_{p^s}\oplus\mathbb{Z}_{p}$ cannot be generated by $\{a,-a\}$ for any $a\in \mathbb{Z}_{p^s}\oplus\mathbb{Z}_{p}$. Thus $k\geq 3$. As $\Gamma$ has the odd order $p^{s+1}$, we assert that is $\Gamma$  antipodal and non-bipartite. By Lemma \ref{lem::imprimitive} and Corollary \ref{cor::DRG_dq}, the antipodal quotient $\overline{\Gamma}$ of $\Gamma$ is a primitive distance-regular  Cayley graph over the cyclic group or the group $\mathbb{Z}_{p^l}\oplus\mathbb{Z}_{p}$ for some $l\in\{1,2,\ldots,s-1\}$. Then it follows from Lemma \ref{lem::cir_DRG}, Corollary \ref{cor::pri_DRG} and Lemma \ref{lem::1} that $\overline{\Gamma}$ is a complete graph, a cycle of prime order, a Payley graph of prime order, or the line graph of a transversal design $TD(r,p)$ with $2\leq r\leq p-1$.  If $\overline{\Gamma}$ is a cycle of prime order, then $\Gamma$ would be a cycle, which is impossible. If $\overline{\Gamma}$ is a Payley graph of prime order, by  Lemma \ref{lem::confer}, we also deduce a contradiction. If $\overline{\Gamma}$ is the line graph of a transversal design $TD(r,p)$ with $2\leq r\leq p-1$,  then $d=4$ or $5$. By Lemma \ref{lem::antipodal cover of line graph}, we assert that $r=2$, and hence  $\overline{\Gamma}$ is the Hamming graph $H(2,p)$. However, by Lemma \ref{lem::anti_Ham}, $H(2,p)$ has no distance-regular antipodal covers for $p>2$, and we obtain a contradiction. 
		Therefore, $\overline{\Gamma}$ is a complete graph, and so $d=2$ or $3$ according to Lemma \ref{lem::imprimitive}. By Lemma \ref{lem::key1}, $d\neq 3$, whence $d=2$. However, complete multipartite graphs are the only antipodal distance-regular graphs with diameter $2$.
		
		This completes the proof.
	\end{proof}

	\section{Further research}\label{section::4}
	
	In this paper, we determine all distance-regular Calyley graphs over $\mathbb{Z}_{p^s}\oplus\mathbb{Z}_{p}$ with $p$ being an odd prime and $s\geq 1$. It is natural to propose the following problem. 
	
	\begin{problem}\label{prob::main1}
		Determine all distance-regular Cayley graphs over $\mathbb{Z}_{n}\oplus\mathbb{Z}_{p}$, where $p$ is a prime factor of $n$.
	\end{problem}
	
	Indeed, for $p=2$,  we can give a solution to Problem \ref{prob::main1} by using a similar method as in \cite{MP07}.  Here we list the main result, and omit its proof.
	
	\begin{theorem}\label{lem:dicyclic_DRG}
		Let $\Gamma$ be a distance-regular Cayley graph over $\mathbb{Z}_n\oplus\mathbb{Z}_2$ with $n>2$ being even. Then $\Gamma$ is distance-regular if and only if it is isomorphic to one of the following graphs:
		\begin{enumerate}[$(i)$]\setlength{\itemsep}{0pt}
			\item the complete graph $K_{2n}$;
			\item the complete multipartite graph $K_{t\times m}$  with $tm=2n$;
			\item the complete bipartite graph without a $1$-factor $K_{n,n} - nK_2$;
			\item the graph $\mathrm{Cay}(\mathbb{Z}_n\oplus\mathbb{Z}_2,(R_0,0)\cup (R_1,1))$, where $R_0=-R_0$ and $R_1=-R_1$ are non-empty subsets of  $1+2\mathbb{Z}_{n}$ such that $(-1+R_0,0)\cup(-1+R_1,1)$ is a non-trivial difference set in $2\mathbb{Z}_n\oplus\mathbb{Z}_2$.
		\end{enumerate}
		In particular, the graph in  $(iv)$ is a non-antipodal bipartite distance-regular graph with  diameter $3$.
	\end{theorem}
	
	\section*{Acknowledgements}
	
	X. Huang is supported by National Natural Science Foundation of China (Grant No. 11901540). L. Lu is supported by National Natural Science Foundation of China (Grant No. 12001544) and  Natural Science Foundation of Hunan Province  (Grant No. 2021JJ40707).

	%\vspace*{4mm}
	
	%\section*{Data availability}
	%Data sharing is not applicable to this article as no datasets were generated or analyzed during the current study.

\end{document}